\RequirePackage{rotating}

\documentclass{article}

\usepackage{arxiv}

\usepackage[utf8]{inputenc} 
\usepackage[T1]{fontenc}    
\usepackage{hyperref}       
\usepackage{url}            
\usepackage{booktabs}       
\usepackage{amsfonts}       
\usepackage{nicefrac}       
\usepackage{microtype}      
\usepackage{amsmath}
\usepackage{cleveref}       
\usepackage{lipsum}         
\usepackage{graphicx}
\usepackage{natbib}
\usepackage{doi}
\usepackage{algorithm}
\usepackage{algpseudocode}
\usepackage{tikz}

\usepackage{graphicx}

\usepackage{natbib}
\usepackage{amsthm}
\usepackage{amssymb}
\usepackage{float}
\usepackage{subcaption}

\usepackage{hyperref}
\usepackage{xcolor}
\usepackage{adjustbox}
\usepackage{longtable}
\usepackage{todonotes}

\usepackage{booktabs}
\usepackage{multirow}

\usepackage{bmpsize}

\newtheorem{observation}{Observation}
\newtheorem{theorem}{Theorem}
\newtheorem{definition}{Definition}
\newtheorem{lemma}{Lemma}

\title{Spectral Outer-Approximation Algorithms for Binary Semidefinite Problems}

\newif\ifuniqueAffiliation
\uniqueAffiliationtrue

\ifuniqueAffiliation 
\author{ {Daniel de~Roux} \\
	Tepper School of Business\\
        Carnegie Mellon University\\
	\texttt{dderoux@andrew.cmu.edu} \\
	\And
	\href{https://orcid.org/0000-0001-6001-1738}{\includegraphics[scale=0.06]{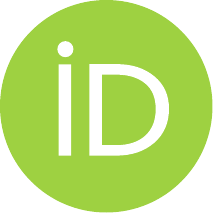}\hspace{1mm}Zedong~Peng} \\
	Davidson School of Chemical Engineering\\
        Purdue University\\
	\texttt{peng372@purdue.edu} \\
	\And
	\href{https://orcid.org/0000-0002-8308-5016}{\includegraphics[scale=0.06]{orcid.pdf}\hspace{1mm}David E.~Bernal Neira} \\
	Davidson School of Chemical Engineering\\
        Purdue University\\
	\texttt{dbernaln@purdue.edu} \\
}
\else
\usepackage{authblk}

\newbox{\orcid}\sbox{\orcid}{\includegraphics[scale=0.06]{orcid.pdf}} 
\author[1]{%
	\href{https://orcid.org/0000-0000-0000-0000}{\usebox{\orcid}\hspace{1mm}David S.~Hippocampus\thanks{\texttt{hippo@cs.cranberry-lemon.edu}}}%
}
\author[1,2]{%
	\href{https://orcid.org/0000-0000-0000-0000}{\usebox{\orcid}\hspace{1mm}Elias D.~Striatum\thanks{\texttt{stariate@ee.mount-sheikh.edu}}}%
}
\affil[1]{Department of Computer Science, Cranberry-Lemon University, Pittsburgh, PA 15213}
\affil[2]{Department of Electrical Engineering, Mount-Sheikh University, Santa Narimana, Levand}
\fi

\renewcommand{\shorttitle}{Spectral Outer-Approximation Algorithms for Binary Semidefinite Problems}

\begin{document}
\maketitle

\begin{abstract}
Integer semidefinite programming (ISDP) has recently gained attention due to its connection to binary quadratically constrained quadratic programs (BQCQPs), which can be exactly reformulated as binary semidefinite programs (BSDPs).
However, it remains unclear whether this reformulation effectively uses existing ISDP solvers to address BQCQPs.
To the best of our knowledge, no specialized ISDP algorithms exploit the unique structure of BSDPs derived from BQCQPs.
This paper proposes a novel spectral outer approximation algorithm tailored for BSDPs derived from BQCQP reformulations. Our approach is inspired by polyhedral and second-order representable regions that outer approximate the feasible set of a semidefinite program relying on a spectral decomposition of a matrix that simultaneously diagonalizes the objective matrix and an aggregation of the constraint matrices.
Computational experiments show that our algorithm is competitive with, and in some cases outperforms, state-of-the-art ISDP solvers such as \texttt{SCIP-SDP} and \texttt{PAJARITO}, highlighting ISDP’s potential for solving BQCQPs.
\end{abstract}


\section{Introduction}
\label{intro}

In recent years, integer semidefinite programming (ISDP) has attracted increasing attention from the optimization community.
ISDPs allow the modeling of mixed-integer non-linear programming (MINLP) optimization problems while leveraging the expressive power of the semidefinite cone.
In particular, semidefinite constraints can naturally capture non-differentiable and combinatorial structures, providing a richer modeling framework compared to traditional mixed-integer linear programs (MILPs).
Moreover, ISDPs enable algorithmic approaches based on conic convex optimization, such as branch-and-bound enhanced by interior-point methods, benefiting from the mature state of semidefinite solvers \cite{dur2021conic}.
This contrasts with global MINLP solvers, which often face numerical instability and struggle with non-differentiable objectives or constraints.

Despite these potential advantages, ISDPs have historically seen limited practical applications.
One reason is that semidefinite programming (SDP) itself emerged primarily as a convex relaxation technique—Shor's relaxation for binary quadratic problems \cite{shor1987quadratic}.
This makes the idea of directly imposing integer constraints within an SDP unusual and, until recently, of relatively limited interest.
Notable exceptions include applications such as truss topology optimization \cite{gally2018framework}, sparse principal component analysis \cite{ li2024exact}, clustering \cite{peng2007approximating,piccialli2022sos}, and the computation of restricted isometry constants \cite{gally2016computing}, but these remain niche compared to the broader world of mixed-integer nonlinear programming.

The relevance of ISDPs has been recently re-energized by the work of \citet{de2023integrality}, which showed that every binary quadratically constrained quadratic program (BQCQP) can be exactly reformulated as a binary semidefinite program (BSDP), a special case of ISDP.
This result bridges ISDPs with the extensive literature on BQCQPs, a central class of combinatorial optimization problems. BQCQPs arise in areas such as combinatorial optimization and computer science \cite{boros1991max,dur2021conic}, machine learning \cite{fortunato2010community,park2017general}, chemical engineering \cite{biegler2004retrospective}, and portfolio optimization \cite{bodnar2013equivalence,deng2013branch}.
Unconstrained binary quadratic problems (BQPs), a special subclass of BQCQPs, have also attracted the attention of the quantum optimization community \cite{domino2022quadratic,mazumder2024five}, with several quantum-inspired algorithms being developed \cite{brown2022copositive}.

This reformulation result motivates renewed attention to ISDPs—not just as abstract mathematical objects, but as computationally meaningful formulations for practically relevant problems.
Importantly, this paper does not claim that ISDP reformulation is the most efficient way to solve all BQCQPs. For some well-structured BQCQPs, such as the quadratic knapsack problem, specialized MIQP solvers like \texttt{Gurobi} often outperform ISDP-based approaches. However, the ISDP perspective is valuable in its own right: it enables the development of tailored ISDP solvers, which could ultimately benefit not only BQCQP-derived problems but also other ISDP applications.

Based on these observations, the natural question arises: how can we design specialized algorithms for ISDPs, particularly ones that exploit the spectral and combinatorial structure inherited from BQCQP reformulations? Addressing this question, this paper proposes a new spectral outer approximation algorithm tailored for ISDPs, with a focus on BSDPs derived from BQCQPs. At the same time, the algorithm is general enough to apply to broader classes of ISDPs beyond those arising from BQCQPs.

\paragraph{Contributions.}
This paper makes the following contributions:

\begin{itemize}
    \item We propose a novel spectral outer approximation algorithm designed for ISDPs, particularly exploiting the spectral structure present in BSDPs arising from BQCQP reformulations.
    \item We show through extensive computational experiments that our algorithm is competitive with, and in some cases outperforms, state-of-the-art ISDP solvers such as \texttt{SCIP-SDP} \cite{gally2018framework} and \texttt{PAJARITO} \cite{lubin2018polyhedral}.
    \item We highlight the growing relevance of ISDPs as a computational tool, contributing to the broader exploration of integer conic optimization.
\end{itemize}

\paragraph{Outline.}  
The remainder of this paper is organized as follows. Section \ref{sec:preliminaries} presents formal definitions of ISDPs and BSDPs, and reviews the BQCQP-to-BSDP reformulation. Section \ref{oa_sec:2} describes the general outer approximation algorithm for ISDPs. Section \ref{oa_sec:3} presents our spectral outer approximation approach. Section \ref{oa_sec:4} presents the computational problems we will  consider and Section \ref{oa_sec:5} provides computational results comparing our method with existing ISDP solvers. Section \ref{oa_sec:6} concludes with final remarks and future research directions.

\subsection{Notation}
 We denote the set of square, real,  $n\times n$ symmetric matrices by $\mathbb{S}^n$ and the set of positive semidefinite matrices by $\mathbb{S}^n_+$. By convention, we assume that semidefinite matrices are always symmetric. We denote the L\"{o}wner order of symmetric matrices by $\succeq$. Hence, the notation 
$X \succeq Y$ means $X-Y \in \mathbb{S}^n_+$. In particular, $X\succeq 0$ indicates that $X$ is positive semidefinite (PSD). For an integer $k \in \mathbb{N}$, $[n]$ denotes the set of natural numbers $\{1,\dots,k\}$.
We denote the cardinality of a set $I$ by $|I|$.
We denote by $e_1,\dots, e_n$ the standard basis of $\mathbb{R}^n$ and the $n\times n$ identity matrix by $I_n$. For a symmetric matrix $W$ we let $\lambda_1(W) \geq \lambda_2(W) \geq \dots \geq \lambda_n$(W) be its eigenvalues. When the matrix is clear from the context, we drop the terms in parentheses and write $\lambda_1\geq\dots\geq\lambda_n$. For $A \in \mathbb{S}^n$ we write $tr(A)$ for the trace of $A$: $tr(A) = \sum_{i=1}^nA_{ii}$. The $\ell_1$ norm of $A$ is given by $\|A\|_1 = \sum_{i,j}|A_{ij}|$. We denote by $\langle\cdot,\cdot\rangle$ the usual Frobenius inner product of two matrices in $\mathbb{S}^n$, recalling that for two matrices $A,B \in \mathbb{S}^n, \ \langle A,B\rangle = tr(A^TB)=tr(AB)$.
We denote by $\Vec{1}$ the vector of all ones in $\mathbb{R}^n$ and by $J$ the matrix of all ones. If $A$ is a matrix, we denote by $diag(A)$ the vector given by the diagonal of $A$. If $u$ is a vector, $diag(u)$ denotes the matrix with $u$ on its diagonal. We
denote by $\mathcal{E}(A)$ an arbitrary orthonormal basis consisting of eigenvectors of $A$. In particular, if $A \in \mathbb{S}^n$ and $\mathcal{E}(A) = \{v_1,\dots,v_n\}$ then we have $A = \sum_{i=1}^n \lambda_iv_iv_i^\top$ \cite{horn2012matrix}.

\section{Preliminaries}
\label{sec:preliminaries}

\subsection{Integer Semidefinite Programs}
An integer semidefinite program (ISDP) is an optimization problem of the form:
\begin{gather}\label{ISDP}\tag{ISDP}
\begin{aligned}
\min_{X \in \mathbb{S}^n }  & \  \langle  C,X\rangle \\
\text{ s.t. } & \langle A_i,X  \rangle = b_i  \quad \forall i \ \in [r],  \\
 & X \succeq 0, \\
& u_{ij} \leq X_{ij} \leq l_{ij} ,\  X_{ij}\in \mathbb{Z},   \quad (i,j)\in L\subseteq [n]\times[n],
\end{aligned}
\end{gather}
with $C,A_i \in \mathbb{S}^n$, $b_i \in \mathbb{R}$ for all $i \in [r]$.
$L\subseteq [n]\times[n]$ indicates the entries of $X$ which are constrained to be integer, and  $u_{ij},l_{ij}$ upper and lower bounds of $X_{ij}$ for $(i,j) \in L$. Notice that if we ignore the integrality constraints, the set of feasible solutions is convex and given by the equations
$ X \succeq 0, \ \langle A_i,X \rangle = b_i  \ \forall i \ \in [r]$.
A particularly important subclass is the binary semidefinite program (BSDP), where
$
X_{ij} \in \{0,1\}, \ \forall (i,j) \in L.
$

\subsection{Binary Quadratically Constrained Quadratic Programs (BQCQP)}
A BQCQP is an optimization problem of the form:
\begin{gather}\label{BQCQP}\tag{BQCQP}
\begin{aligned}
\min_{x \in \{0,1\}^n} & \  x^\top C x + 2d_0^\top x \\
\text{s.t.} & \ x^\top A_i x + 2d_i^\top x \leq b_i \quad \forall i \in [r], \\
& \ Dx = t,
\end{aligned}
\end{gather}
where $C, A_i \in \mathbb{S}^n$, $d_i \in \mathbb{R}^n$, $D \in \mathbb{R}^{q\times n}$, and $t \in \mathbb{R}^q$. Such problems capture a wide variety of practical applications, as discussed in Section \ref{intro}.

\subsection{Exact Reformulation as a Binary Semidefinite Program (BSDP)}
The key result from \cite{de2023integrality} is that any BQCQP can be exactly reformulated as binary semidefinite programs. Formally:

\begin{theorem}[Theorem 9 of \cite{de2023integrality}]
Let $C,A_i \in \mathbb{S}^n, d_0,d_i \in \mathbb{R}^n$, $b_i \in \mathbb{R}, \ \forall i \in [r]$ and $D \in \mathbb{R}^{q \times n }, t_i \in \mathbb{R}, \forall \ i \in [q]$, where $r,q \in \mathbb{N}$. The following semidefinite binary program is equivalent to \eqref{BQCQP}.
\begin{gather}\label{BSDP}\tag{BSDP}
\begin{aligned}
\min_{X, x } \ & \left \langle C,X \right \rangle + 2d_0^\top x \\
\text{s.t.} \ & \langle A_i,X \rangle + d_i^\top x \leq b_i \quad  \ i \in [r], \\
& D x  = t, \\
& X-xx^\top \succeq 0, \\
& Diag(X)=x, \\
& x \in \{0,1\}^n.
\end{aligned}
\end{gather}
    
\end{theorem}

The constraint $X-xx^\top \succeq 0$ is usually written as the equivalent constraint $\begin{bmatrix} X &  & x^\top  \\ x & & 1\end{bmatrix} \succeq  0$.

\section{Outer approximation algorithms for Integer semidefinite problems}\label{oa_sec:2}

In this section, we introduce an outer approximation algorithm for integer semidefinite optimization, as outlined in Algorithm \ref{alg_pajarito}. This algorithm is introduced in the more general setting of mixed-integer conic optimization by
 \citet{lubin2018polyhedral}. A polyhedral outer approximation of the set of positive semidefinite matrices $\mathbb{S}^n$ is defined as follows.

\begin{definition}
    A polyhedron $P$ is an outer approximation of $\mathbb{S}^n$ if $P$ equals the intersection of a finite number of half-spaces and contains $\mathbb{S}^n$.
\end{definition}

Outer approximations of the semidefinite cone are obtained by fixing a finite set of positive semidefinite matrices since $\mathbb{S}^n_+$ is self-dual. That is, a polyhedral set $P = \{X \in \mathbb{S}^n : \left \langle T_i, X \right \rangle \geq 0 ,\  i \in [q] \}$, where $T_1,\dots, T_q \in \mathbb{S}^n_+$ is an outer approximation of $\{X \in \mathbb{S}^n : X \succeq 0\}$. Because we will need to keep track of the matrices $T_i$, we set $\mathcal{T}:= \{T_1,\dots,T_q\}$ and let

\[
P_{\mathcal{T}} :=  \{X \in \mathbb{S}^n : \left \langle T, X \right \rangle \geq 0  \ \forall \ T \in \mathcal{T}\}.
\]

Consider a problem of the form of \eqref{ISDP}. Let $P_{\mathcal{T}}$ be a polyhedral outer approximation of $\mathbb{S}^n_+$. Then, the problem
\begin{gather}\label{oasdp}\tag{OASDP($P_{\mathcal{T}}$)}
\begin{aligned}
\min_{X \in \mathbb{S}^n }  \ & \langle  C,X\rangle \\
\text{ s.t. }\ & \langle A_i,X \rangle = b_i \quad \forall i \ \in [r], \\ 
&  X \in P_{\mathcal{T}}, \\
&  s_{ij} \leq X_{ij} \leq l_{ij} , \ X_{ij}\in \mathbb{Z} \quad (i,j)\in L\subseteq [n]\times[n]
\end{aligned}
\end{gather}
is called the linear outer approximation problem of \eqref{ISDP}. Notice that this problem is a mixed-integer linear problem. Since any feasible solution to \eqref{ISDP} is feasible to \eqref{oasdp} as well, this latter problem is a relaxation of the former, and the optimal value of \eqref{oasdp} provides a lower bound to \eqref{ISDP}.

Notice that if we fix the entries of $X$ to have a specific integer value in program \eqref{ISDP}, the program reduces to a semidefinite optimization problem. More formally,
Let $X^L$ be a matrix with $X^L_{ij} \in \mathbb{Z}$ and $u_{ij}\leq X^L_{ij} \leq l_{ij}$ for all $(i,j) \in L \subseteq [n]\times [n]$. Consider the following optimization program
\begin{gather}\label{inner_sdp}\tag{SDP($X^L$)}
\begin{aligned}
\min_{X \in \mathbb{S}^n }  & \  \langle  C,X\rangle \\
\text{s.t. } & \langle A_i,X \rangle = b_i  \quad \forall i \ \in [r],  \\
& \ X \succeq 0, \\ 
& X_{ij} = X^L_{ij} \quad \forall (i,j) \in L.
\end{aligned}
\end{gather}

Since we have fixed the integer coordinates to take a specific value, this program is a (convex) positive semidefinite optimization problem. Furthermore, if $X'$ is an optimal solution to this problem, it is feasible for \eqref{ISDP} and therefore provides a corresponding upper bound. 

In summary, the outer approximation algorithm for \eqref{ISDP} first computes an optimal solution $\Hat{X}$ to \eqref{oasdp}. This solution necessarily has integer values in the entries specified by $L$, and the algorithm proceeds to solve problem $SDP(\Hat{X})$. That is problem \eqref{inner_sdp} where the entries in $L$ match the entries of $\Hat{X}$ in $L$. If all problems are solvable and $X^*$ denotes an optimal solution of \eqref{ISDP}, then the following key inequalities hold

\begin{equation}
\left \langle C, \Hat{X} \right \rangle \leq \left \langle C, X^* \right \rangle \leq \left \langle C, X' \right \rangle.
\end{equation}

If the values $\left \langle C, \Hat{X} \right \rangle$ and 
$\left \langle C, X' \right \rangle$ are equal, then $X'$ is optimal for \eqref{ISDP} and the algorithm terminates. If not, the algorithm proceeds by updating the outer approximation $P_\mathcal{T}$ to a tighter approximation by the addition of valid linear constraints, and \eqref{oasdp} is resolved. This process yields a non-decreasing sequence of lower bounds. We give the full algorithm steps in \ref{alg_pajarito}. We point out that since the number of possible integer assignments for the coordinates in $L$ is finite, Algorithm \ref{alg_pajarito} is guaranteed to terminate in a finite number of steps if there is no repetition of assignments of the integer variables.

\begin{algorithm}
\caption{Outer Approximation (SDP)}
\label{alg_pajarito}
\begin{algorithmic}[1]
\State Fix a tolerance $\varepsilon>0$.  
Set $\mathcal{T} = \{(e_i\pm e_j)(e_i\pm e_j)^\top: i,j \in [n]\}$.
 \State Solve problem OASDP($P_\mathcal{T}$) finding a minimizer $\Hat{X}$. 
\State Solve problem $SDP(\Hat{X})$, finding a minimizer $X'$.
\If{$ \left |  \left \langle C, X' \right  \rangle -  \left \langle C,   \Hat{X} \right \rangle \right | > \varepsilon$}
\State Find a dual optimal dual solution $S'$ of program $DSDP(\Hat{X})$. 
\State Set $\mathcal{T} = \mathcal{T} \cup \{S'\}$. Go to step $2$.
\EndIf
\Return $X'$.
\end{algorithmic}
\end{algorithm}

Lemma \ref{lem:finiteSDP} guarantees the finite convergence of Algorithm \ref{alg_pajarito}. To state this result, we need first to specify the dual of problem \eqref{inner_sdp}.

\begin{observation}
The dual of problem \eqref{inner_sdp} is given by    
\begin{gather}\label{Dual_OA_sdp}\tag{DSDP($X^L$)}
    \begin{aligned}
\max_{\gamma \in \mathbb{S}^n, S \in \mathbb{S}^n_+ , y\in \mathbb{R}^n} & \ \  \  b^\top y  + \left \langle \gamma, X^L \right \rangle \\
\text{s.t.} \quad & \gamma_{i,j} = 0 \quad  \forall (i,j)\in  [n]\times [n] \setminus I , \\ 
& C - \sum_{i=1}^r  A_iy_i - \gamma = S  \succeq 0.
\end{aligned}
\end{gather}

\end{observation}

The key to obtaining appropriate hyperplanes to update the set $P_\mathcal{T}$ that implies the convergence of the algorithm in finite time is conic \textit{duality}.

\begin{lemma}\label{lem:finiteSDP}

Let $X^L$ be fixed and denote by $Z_{X^L}$ the optimal value of problem \eqref{inner_sdp}, with corresponding minimizer $X'$. Suppose that strong duality holds between the pair of problems 
\eqref{Dual_OA_sdp} and \eqref{inner_sdp}. 
Let $S'$ be optimal for the latter program. Set $\mathcal{T} = \{S'\}$ so that
\[
P_{\mathcal{T}} = \{X \in \mathbb{S}^n:\left \langle X,S'\right \rangle \geq 0\}. 
\]
Let $\Hat{X}$ be such that
$\langle A_i,\Hat{X}  \rangle = b_i \ \forall i \in [m]$, 
$\langle \Hat{X}, S' \rangle \geq 0$ and such that $\Hat{X}_{ij} = X^L_{ij}$ for all $(i,j) \in L$ . Then,
$\langle  C,\Hat{X} \rangle \geq Z_{X^L}$.
In addition, if $\Hat{X}$ is optimal for program \eqref{oasdp} - or in other words if the outer approximation \eqref{oasdp} returns a matrix with integer part equal to $X^L$-, then $X'$ is global optimal for \eqref{ISDP} and the outer approximation algorithm terminates.
\end{lemma}

\begin{proof}

First, observe that
\begin{equation}
 \begin{aligned}
     0 \leq \langle \Hat{X}, S' \rangle & = \langle \Hat{X}, C-\sum_{i=1}^r A_iy_i-\gamma \rangle 
    =   \langle \Hat{X}, C \rangle - \langle \Hat{X}, \sum_{i=1}^r A_iy_i-\gamma \rangle \\
   &  =  \langle \Hat{X}, C \rangle -  \sum_{i=1}^r y_i \langle \Hat{X},A_i \rangle - \langle \Hat{X},\gamma \rangle\\
   & =  \langle \Hat{X}, C \rangle -  \sum_{i=1}^r y_ib_i - \langle \Hat{X},\gamma \rangle \\
   & = \langle \Hat{X}, C \rangle -  \sum_{i=1}^r y_ib_i - \langle X^L,\gamma \rangle =  \langle \Hat{X}, C \rangle - Z_{X^L}.
  \end{aligned}
\end{equation}

The last equation is valid because $\gamma$ is zero for the $(i,j)$ entries not in $L$ and because $\Hat{X}$ matches $X^L$ in those entries. Hence, we derive 
$\langle C,\Hat{X} \rangle \geq Z_{X^L}$.

To conclude, let $OPT$ denote the optimal value of \eqref{ISDP}. Observe that since program \eqref{oasdp} is a relaxation of \eqref{ISDP}, we have $ \langle C,\Hat{X} \rangle \leq OPT $. Now, $Z_{X^L}$ is the optimal value of program \eqref{inner_sdp}  whose optimizer is feasible to program \eqref{ISDP} so that we have $OPT \leq Z_{X^L} $. All in all, we get the inequalities

\[\langle C,\Hat{X} \rangle \leq OPT  \leq Z_{X^L} \leq \langle C,\Hat{X} \rangle\]
which imply that $Z_{X^L} = OPT$. Therefore the gap with outer approximation relaxation is $0$, and $X'$ is optimal for \eqref{ISDP}.
\end{proof}

The main consequence of this lemma is that the outer approximation algorithm will not cycle through integer solutions. Indeed, if an integer solution is repeated in program \eqref{inner_sdp}, then the algorithm will terminate in the next step by proving a gap of $0$ between the inner and outer approximations. This algorithm is guaranteed to terminate under mild assumptions but in the worst case it might need to solve an exponential number of subproblems. We point out that the \textit{initialization} of $\mathcal{T}$ to the set 
$ \{(e_i\pm e_j)(e_i\pm e_j)^\top: i,j \in [n]\}$ amounts to the linear constraints on $X$ given by
$
X_{ii} + X_{jj} \geq 2|X_{i,j}| \ \forall i,j \in [n],
$
which is necessary for positive semidefinitness.

\section{Refining outer approximations}\label{oa_sec:3}

The efficiency of the outer approximation algorithm for integer semidefinite programs primarily depends on two factors: the speed of solving integer subproblems and the quality of the polyhedral approximation. Recent advances in commercial solvers such as \texttt{Gurobi} have significantly improved solutions to mixed-integer linear and second-order problems. Thus, further improvements in outer approximation methods are more likely to come from enhancing the quality of polyhedral approximations.

In \cite{de2023instance}, the authors tackle the issue of constructing effective polyhedral approximations to the feasible sets of semidefinite optimization programs. The key insight is to leverage the spectral properties of both the objective coefficient matrix and the matrices defining the semidefinite constraints. A polyhedral approximation is considered "good" if solving the associated linear relaxation yields objective values close to those of the original semidefinite program.

In this section, we apply this approach to binary semidefinite problems in the form \eqref{BSDP}. Due to the structure of the constraint $X - xx^\top \succeq 0$, it is convenient to consider the equivalent constraint
\[
\begin{bmatrix} X & x^\top \\ x & 1 \end{bmatrix}\succeq 0,
\]
thus allowing us to adapt techniques from \cite{de2023instance}.

Consider the standard semidefinite optimization problem
\begin{equation}
\begin{aligned}
\min_{X\in \mathbb{S}^{ n}} & \langle C,X\rangle \quad \\
\text{s.t. } & \langle A_i,X \rangle = b_i \quad \forall i \in [r], \\
& X \succeq 0, 
\end{aligned}
\end{equation}
with its dual
\begin{equation}
\begin{aligned}
\max_{y\in \mathbb{R}^{n}} & \ b^\top y \\
\text{s.t. } & C - \sum_{i=1}^r y_i A_i \succeq 0.
\end{aligned}
\end{equation}

Assuming strong duality and strict feasibility, we consider linear relaxations of this SDP given by
\begin{equation}
\begin{aligned}
\min_{X\in \mathbb{S}^{n}} \ & \langle C,X\rangle \\
\text{s.t. } & \langle A_i,X \rangle = b_i \quad \forall i \in [r], \\
& v^\top X v \geq 0 \quad \forall v \in \mathcal{S},
\end{aligned}
\end{equation}
for some finite set $\mathcal{S}$. The authors of \cite{de2023instance} demonstrated that choosing $\mathcal{S}$ based on eigenvectors of certain dual feasible solutions provides effective polyhedral approximations.
Specifically, they showed that if the matrices $\{C, A_1, \dots, A_r\}$ are simultaneously diagonalizable (SD), then the linear relaxation matches exactly the SDP solution. Although full simultaneous diagonalizability is rare, a useful relaxation is introduced: requiring $C$ to be simultaneously diagonalizable with a suitable aggregation of the constraint matrices, possibly including the identity matrix. Under these relaxed conditions, the authors provided conditions ensuring that the linear relaxation still closely approximates the original SDP.

Crucially, the task of identifying simultaneously diagonalizable aggregations reduces to finding matrices that commute with $C$, due to the well-known characterization of simultaneous diagonalizability:

\emph{A set of symmetric matrices is simultaneously diagonalizable if and only if all matrices in the set commute pairwise.}

For $y\in \mathbb{R}^r$ and matrices $A_1,\dots,A_r \in \mathbb{S}^n$ denote by $\mathcal{A}(y)$ the aggregation 
$\sum_{i=1}^ry_iA_i$. Finding aggregations that commute with the objective can thus be formulated and solved via a linear program:
\begin{equation}
\label{prog_commute}
\begin{aligned}
    \min_{y\in \mathbb{R}^{n}} \ & f(y) \\
    \text{s.t. } & C\mathcal{A}(y) = \mathcal{A}(y)C,
\end{aligned}
\end{equation}

where $f(y)$ is an arbitrary linear function, and notably, the zero matrix always serves as a feasible solution. We refer the reader to 
\cite{de2023instance} for the details. 

\subsection{Second-order strengthening}

We recall that the second-order cone $\mathcal{L}^{1+n}$ is given by $
\mathcal{L}^{1+n} = \{(r,t):\mathbb{R}^{1+n}:r \geq \|t\|^2_2\}.
$
In the particular case of an integer semidefinite optimization problem, second-order necessary conditions for positive semidefinitness can be imposed, resulting in second-order integer problems for the outer approximation step in Algorithm \ref{alg_pajarito}, rather than integer linear problems.
\cite{coey2020outer} takes this idea further and proposes a version of \texttt{PAJARITO} using a second-order cone outer approximation for the outer approximation step. The \emph{rotated second order cone} $\mathcal{V}^{2+n}$ is given by $
\mathcal{V}^{2+n} = (r,s,t) \in \mathbb{R}^{2+n}: r,s\geq \ 0, 2rs \geq \|t\|_2^2. 
$
This cone is self-dual and can be obtained as an invertible linear transformation of the standard second-order cone $\mathcal{L}$ as $(r,s,t) \in \mathcal{V}^{n+2}$ if and only if $(r+s,r-s,\sqrt{2}t_1,\dots, \sqrt{2}t_n ) \in \mathcal{L}^{2+n}$. One can check that given $X\succeq 0$ the \emph{rotated second-order cone constraints}
$
(X_{ii},X_{jj}, \sqrt{2}X_{ij}) \in \mathcal{V}^3 
$
are valid for each $i$ and $j$. This corresponds to saying that the $2 $ by $2$ minors of $X$ are positive semidefinite. 
Equivalent cuts are also described in  \cite{bertsimas2020polyhedral} where the authors mention that if $X$ is positive semidefinite, then $X$ satisfies

\begin{equation}\label{eq:soc_cuts}
\left\| \begin{pmatrix}2X_{i,j} \\ X_{i,i}-X_{j,j} \end{pmatrix} \right\|_2 \leq X_{i,i}+X_{j,j}, \ \forall i \in [n], \  \forall j \in [n].
\end{equation}

These cuts are also mentioned in \cite{wang2021polyhedral}, and
in fact, all of them can be derived using an alternative version of the Schur Complement Lemma presented in \cite{kim2003second}. We briefly mention this idea in Subsection \ref{subsec:dissagregations}. 
However, it is not clear that enforcing the PSD constraints on $2$ by $2$ minors is necessarily beneficial. Although they provide a tighter approximation, the additional $\frac{n^2-n}{2}$ cuts added result in a heavy burden on the integer second-order solver. Therefore, if second-order cone cuts are to be added, they must be few and significantly improve the quality of the approximation. In what follows, we describe a derivation of such cuts using the specific structure that binary SDPs arising from binary QCQPs have. \emph{Quadratically constrained quadratic problems} are problems of the form \ref{qcqp_prob} and their corresponding Shor semidefinite relaxation is given by \eqref{qcqp_prob_shor}

\begin{equation} 
\label{qcqp_prob}
\begin{aligned}
\min_{x } & \  x^\top C x + 2d_0^\top x \\ 
\text{s.t. }& \ x^\top A_i x + 2d_i^\top x \leq b_i \quad \forall i \in [r]
\end{aligned}
\end{equation}

\begin{equation} 
\label{qcqp_prob_shor}
\begin{aligned}
\inf_{x \in \mathbb{R}^n, X\in \mathbb{S}^n} & \left \langle C,X \right \rangle + d_0^\top x  \\
 \text{s.t. } & \left \langle A_i,X \right \rangle +  d_i^\top x\leq  b_i \quad \forall i \in [r], \\ 
& X-xx^\top \succeq 0.
\end{aligned}     
\end{equation}

By fixing a finite set $\mathcal{S}\subseteq \mathbb{R}^n$, this latter problem can be relaxed to the second-order cone problem 

\begin{equation}\label{shor_soc}
\begin{aligned}
&\inf_{x \in \mathbb{R}^n, X\in \mathbb{S}^n}  \left \langle C,X \right \rangle + d_0^\top x  \\
   & \text{s.t. }   \left \langle A_i,X \right \rangle +  d_i^\top x\leq  \ b_i \quad \forall i \in [r], \\ 
 &  \ \ \ \ v^\top \left  (X-xx^\top\right)v \geq 0 \quad \forall   v\in \mathcal{S}.
\end{aligned}     
\end{equation}

Notice that the cuts $ v^\top \left  (X-xx^\top\right)v \geq 0$ are tailored to the special structure of program \eqref{BSDP} and are second-order cuts, resulting in a much tighter approximation of the convex region $X-xx^\top \succeq 0$ than the polyhedral alternative.

We now propose a second-order cone outer approximation algorithm to solve problem \eqref{BSDP}. In this setting, the ambient space of the positive semidefinite matrices considered is $\mathbb{S}^{n+1}$ since our variable matrix is $\begin{bmatrix}
X & x \\
x^\top & 1 
\end{bmatrix} \succeq 0$. Let $\mathcal{S}$ be a finite subset of $\mathbb{R}^n$. Let $\mathcal{T} = \{T_1,\dots,T_q\}\subseteq \mathbb{S}^{n+1}_+$ and define 
$P_{\mathcal{T}}:= \{X \in \mathbb{S}^{n+1}: \langle X, T \rangle \geq 0 \ \forall T \in \mathcal{T} \} $. Consider the second-order relaxation of \eqref{BSDP} given by

\begin{gather}\label{soc_bsdp}\tag{SOC($\mathcal{S},\mathcal{T}$)}
\begin{aligned}
\min_{X, x } \ & \left \langle C,X \right \rangle  + 2d_0^\top x \\
\text{ s.t. } & \langle A_i,X \rangle + d_i^\top x \leq b_i \quad \forall \ i \in [r], \\
&   D x  = t, \\
&  Diag(X)=x, \\
& x \in \{0,1\}^n,\\
& v^\top \left  (X-xx^\top\right)v  \geq 0 \quad \forall v\in \mathcal{S}, \\
& \begin{bmatrix}
X & x \\
x^\top & 1 
\end{bmatrix} \in P_{\mathcal{T}}.
\end{aligned}
\end{gather}
    
With this program at hand, we can introduce the \emph{spectral second order outer approximation algorithm}.

\begin{algorithm}
\caption{Spectral-second-order Outer-Approximation}
\label{alg_spec_soc_OA}
\begin{algorithmic}[1] 

\State Fix a tolerance $\varepsilon>0$. Find $q^1 \in \mathbb{R}^n$ such that $\sum_{i=1}^r q^1_iA_i = I$.  
\State Use program \eqref{prog_commute} to find $q^2$ with support disjoint from $q_1$ such that the matrices $C$ and $\sum_{i=1}^r A_iq_i^2$ commute. Let $U$ be a matrix that simultaneously diagonalizes $C$ and $\sum_{i=1}^r A_iq_i^2$. Denote its columns by $v_1,\dots, v_n$.
\State Set $\mathcal{S} = \{ v_1,\dots, v_n \}$. 
Set $\mathcal{T} = \emptyset.$
\State Solve problem \eqref{soc_bsdp} finding a minimizer $\Bar{X}$. 
\State Solve problem $SDP(\Bar{X})$, finding a minimizer $\Hat{X}$.
\If{$ \left |  \left \langle C, \Hat{X} \right  \rangle -  \left \langle C,   \Bar{X} \right \rangle \right | > \varepsilon$}
\State Find a dual optimal dual solution $\Hat{S}$ of program $DSDP(\Bar{X})$. 
\State Set $\mathcal{T} = \mathcal{T} \cup \{\Hat{S}\}$. Go to step $4$.
\EndIf.
\State Return $\Hat{X}$.
\end{algorithmic}
\end{algorithm}

In this algorithm, we update the outer approximation by adding linear constraints coming from a dual optimal solution of 
$DSDP(\Bar{X}^L)$. By Lemma \ref{lem:finiteSDP}, this guarantees the termination of the algorithm. However, since we are now dealing with a second-order cone integer program, one may think of adding second-order cuts to strengthen the outer approximation further. In other words, it might be worthwhile to update the set $\mathcal{S}$. We describe this process in the next subsection.

An interesting variant of this algorithm consists of using so-called \emph{lazy constraints}. In fact, notice that Algorithm \ref{alg_spec_soc_OA} solves an integer second-order cone program and hence explores a Branch-and-Bound tree at each iteration. Instead, it seems worthwhile to cut solutions that are not positive semidefinite in each of the nodes by imposing constraints of the form $v^\top X v \geq 0$ dynamically. Such dynamic constraint generation is known as lazy constraint callback in the optimization literature.
Given an instance of problem \eqref{soc_bsdp} with $\mathcal{T}= \emptyset$ and $\mathcal{S} = \{v_1,\dots,v_n\}$ where the $v_i$ are defined as in Algorithm \ref{alg_spec_soc_OA}, the variant algorithm maintains a pool of lazy constraints $\mathcal{LC}$ that is initially empty. Then, a Branch-and-Bound procedure is initiated. Whenever an optimal solution $\Hat{X}$ of a branch is found, the algorithm tests whether $u^\top X u \geq 0$ where $u$ ranges over the elements $\mathcal{LC}$. If all of these inequalities are satisfied but $\Hat{X}$ is not PSD, the algorithm computes an eigenvector $w$ corresponding to the least eigenvalue of $\Hat{X}$ and updates $\mathcal{LC}$ to $ \mathcal{LC}\cup \{w\} $. This procedure is iterated until optimality is proven by the Branch-and-Bound procedure. We refer to this algorithm as the Spectral-Lazy-second-order-B\&C.

\begin{algorithm}
\caption{Spectral-Lazy-second-order-B\&C}
\label{alg_lazy_soc}
\begin{algorithmic}[1]

\State Fix a tolerance $\varepsilon>0$. Find $q^1 \in \mathbb{R}^n$ such that $\sum_{i=1}^r q^1_iA_i = I$.  
\State Use program \eqref{prog_commute} to find $q^2$ with support disjoint from $q_1$ such that the matrices $C$ and $\sum_{i=1}^r A_iq_i^2$ commute. Let $U$ be a matrix that simultaneously diagonalizes $C$ and $\sum_{i=1}^r A_iq_i^2$. Denote its columns by $v_1,\dots, v_n$. Set $\mathcal{S} = \{ v_1,\dots, v_n \}$. 
\State Start (or continue) the Branch-and-Bound procedure to solve problem \eqref{soc_bsdp}. 
\State Whenever a feasible solution $\Hat{X}$ is found go to $5$. 
\State if $\lambda_n(\Hat{X})<-\varepsilon$, add the constraint $w^\top Xw \geq 0$ where $w$ is a eigenvector corresponding to $\lambda_n(\Hat{X})$ to program  \eqref{soc_bsdp}.
\State Go to $3$.
\Return $\Hat{X}$.
\end{algorithmic}
\end{algorithm}

\subsection{Cut disaggregation}\label{subsec:dissagregations}

At a given iteration, the outer approximation algorithm 
solves problem $SDP(\Bar{X})$ and finds an optimal dual solution $S'$ of problem $DSDP(\Hat{X})$. If the gap between the objective of the outer approximation integral program and the inner semidefinite program with fixed integer values exceeds a threshold $\varepsilon$, the algorithm iterates by refining the outer approximation, adding the constraints $\langle X, S'\rangle \geq 0$ which guarantees that the outer approximation algorithms terminate in finite time. This strategy can be improved by adding cuts that are implied by the positive semidefinitness of $X$ and that, in turn, imply $\langle X, S'\rangle \geq 0$. The following desegregation of cuts are suggested in \cite{de2023instance} and in \cite{coey2020outer}.

\begin{observation}
Let $S \in \mathbb{S}^n_+$ be a positive semidefinite matrix. We have that $S = \sum_{j=1}^n \lambda_jv_jv_j^\top$ where $\lambda_j, j\in [n]$ are the eigenvalues of $S'$ and the vector $v_j$ is a eigenvector of $S$ corresponding to $\lambda_j$ for each $j\in [n]$. Then, \ $ \langle X, v_jv_j^\top \rangle = v_j^\top X v_j \geq 0 \ \forall \ j \in [n] \text{, implies } \langle X,S \rangle \geq 0.$
 
\end{observation}

These constraints are linear in $X$ and, therefore, can be added to Algorithm \ref{alg_spec_soc_OA} in step $8$.
Perhaps more interestingly, the disaggregation of $S = \sum_{j=1}^n \lambda_jv_jv_j^\top$ can also be used to impose second-order cone constraints directly related to the structure of the semidefinite matrix  $\begin{bmatrix}
X & x \\
x^\top & 1 
\end{bmatrix}$. Notice the optimal dual variable $S^*$ obtained in step $7$ of Algorithm \ref{alg_spec_soc_OA} is of dimension $n+1 \times n+1$.

\begin{lemma}\label{lem_soc_valid}
Suppose that the matrix $\begin{bmatrix}
X & x \\
x^\top & 1 
\end{bmatrix} \in \mathbb{S}^{n+1} $ is positive semidefinite, or equivalently $X-xx^\top \succeq 0$. Let $S = \sum_{j=1}^{n+1} \lambda_j v_jv_j^\top \succeq 0$. For each $j \in [n+1]$ denote by $z_j$ the $n+1$st entry of $v_j$ and by $w_j \in \mathbb{R}^n$ the vector $v_j$ restricted to its first $n$ entries. Furthermore, suppose that for each $j \in [n+1]$ the equation $
w_j^\top X w_j \geq (w_j^\top x)^2$
holds. Then, it follows that $
\left \langle  \begin{bmatrix}
X & x \\
x^\top & 1 
\end{bmatrix} , S \right \rangle \geq 0.$

\end{lemma}
\begin{proof}

Observe that for each $j \in [n+1]$ we have $
(w_j^\top x + z_j)^2 = (w_j^\top x )^2 + 2w_j^\top x z_j + z_j^2 \geq 0.$
This implies that $2(w_j^\top x)z_j + z^2 \geq -(w_j^\top x)^2$.
Now, we have that $w_j^\top X w_j \geq (w^\top x)^2$ because $X-xx^\top \succeq 0$. Adding the two previous equations yields

$
w_j^\top X w_j + 2(w_j^\top x)z_j + z_j^2 \geq  -(w^\top x)^2 + (w^\top x)^2 = 0.
$
To conclude, observe that

\[
\left \langle  \begin{bmatrix}
X & x \\
x^\top & 1 
\end{bmatrix}   , v_jv_j^\top  \right \rangle = w_j^\top X w_j + 2(w_j^\top x)z_j + z_j^2,
\]
so multiplying by $\lambda_j \geq 0$ (since $S\succeq 0$) and adding over $j$ gives the result.

\end{proof}

To use the disaggregation of $S' =\sum_{j=1}^{n+1} \lambda_j v_jv_j^\top$ we can add the cuts $w_j^\top X w_j \geq (w_jz_j)^2 $ for each $j\in [n+1]$ in step $8$ of algorithm \ref{alg_spec_soc_OA}. By our previous lemma, this ensures that the conditions of Lemma \ref{lem:finiteSDP} are satisfied.

In the implementation proposed in \cite{coey2020outer}, the authors propose an alternate way of generating second-order cuts using the disaggregation of $S'$ based on the ideas of 
\cite{kim2003second}, which proposes a set of second-order quadratic constraints implied by positive semidefinitness. We briefly present their idea for comparison. Let $X \in \mathbb{S}^n$ be positive semidefinite. Fix $i \in [n]$ and let $w \in \mathbb{R}^n$ be an arbitrary vector. Let $w_i$ be the $i-th$ entry of $w$, and $\Bar{w} \in \mathbb{R}^{n-1}$ be the vector obtained by removing the $i-th$ entry of the vector $w$. Let $X_{ii}$ be the $i-th$ diagonal entry of a matrix $X$.

Let $u \in \mathbb{R}^{n-1}$ be the vector obtained by removing the $i-th$ entry of the $i-th$ row of $X$, i.e., by removing $X_{ii}$ from the $i-th$ row. Let $\Bar{X}\in \mathbb{S}^{n-1}$ be the matrix obtained by removing from $X$ its $i-th$ column and row. In \cite{kim2003second}, it is shown that $X\succeq 0$ if $X_{ii} \geq 0 $ and

$
\Bar{X} \succeq 0 \ \text{ and } (X_{i,i}) \Bar{X} - uu^\top \geq 0.
$Furthermore,  observe that the constraint
\begin{equation}\label{constr_rsoc}
(X_{ii}, \Bar{w}^\top \Bar{X}  \Bar{w}, \sqrt{2} \Bar{w}^\top u ) \in \mathcal{V}^3
\end{equation}
is equivalent to 
\begin{equation}
w^\top  \Bar{X} w \geq 0 \text{ and } w^\top (X_{ii} \Bar{X}) w \geq (w^\top u)^2.
\end{equation}

The condition $ \Bar{X} \succeq 0 $ implies $w^\top  \Bar{X} w \geq 0$ and it is direct to check that 
$(X_{i,i}) \Bar{X} - uu^\top \geq 0$ implies that 
$ \Bar{w}^\top (X_{ii} \Bar{X})  \Bar{w} \geq ( \Bar{w}^\top u)^2.$
Hence, the constraint \eqref{constr_rsoc} is implied by $X \succeq 0$ and can be used to obtain a stronger, second-order cone outer approximation to the positive semidefinite cone.  Setting $w = e_i\pm e_j, i>j \in [n]$ results in the cuts 
$\mathcal{V}^{2+n} = (r,s,t) \in \mathbb{R}^{2+n}: r,s\geq \ 0, 2rs \geq \|t\|_2^2 $ proposed in \citet{coey2020outer}. To generate cuts using the disaggregation of $S'$, the authors heuristically set $i$ to be the index of the largest absolute entry of a vector $v_j$ and add the cut \eqref{constr_rsoc} by setting $w=v_j$.

\section{Applications to binary semidefinite programs}
\label{oa_sec:4}

In this section, we introduce two binary semidefinite problems to evaluate our algorithms. In Section \ref{oa_sec:5}, we test the different algorithms considered for the problems below.

\subsection{Cardinality-constrained Boolean least squares}

The cardinality-constrained Boolean least squares is a problem of the following form:

\begin{equation}
\begin{aligned}
\min_{x \in \mathbb{R}^n}  &\|Ax-b\|^2_2 \\
\text{s.t.} & \sum_{i=1}^n x_i \leq k, \\
& x \in \{0,1\}^n
\end{aligned}
\end{equation}

where $A \in \mathbb{R}^{m \times n}$ and $b \in \mathbb{R}^m$, $k\in \mathbb{N}$. This problem appears in digital communications in the setting of maximum likelihood estimation of digital signals \cite{park2017general}  By writing $x_i^2 = x_i$, this problem can be reformulated exactly as the following BSDP:
\begin{gather}\label{BLS_SDP}\tag{BLSSDP}
\begin{aligned}
\inf_{x \in \mathbb{R}^n, X\in \mathbb{S}^n} & \left \langle A^\top A,X \right \rangle - 2b^\top A x + b^\top b  \\
    \text{s.t. } & Diag(X) = x,  \\
    & \sum_{i=1}^n x_i \leq k,\\
    & X-xx^\top \succeq 0, \\
    & x \in \{0,1\}.
\end{aligned}
\end{gather}

This problem is closely related to the computation of
\emph{restricted isometry constants}. These quantities are relevant in the context of compressed sensing and are NP-hard to compute.
In \cite{gally2016computing}, Gally and Pfetsch propose a mixed-integer semidefinite program to compute them. \cite{kobayashi2020branch} uses the same problem to test their proposed algorithms  In contrast to our work, \cite{gally2016computing,kobayashi2020branch} let the $x$ variables take arbitrary real values while we fix them to be binary.

\subsection{Quadratic knapsack problem}

We next consider the \emph{Quadratic Knapsack problem } \cite{pisinger2007quadratic} where given $w \in \mathbb{R}^n$, $C \in \mathbb{S}^n$, $c \in \mathbb{R}_+$ can be written as in \ref{qkp}. Its exact reformulation is given in \ref{qkp_sdp}.

\begin{gather}\label{qkp}\tag{QK}
\begin{aligned}
\max_{x \in \mathbb{R}^n} & \  x^\top C x  \\
\text{s.t.} & \ \sum_{j=1}^k w_j x_j  \leq c, \\ 
& \  x \in \{0,1\}^n 
\end{aligned}
\end{gather}

\begin{gather}\label{qkp_sdp}\tag{QKSDP}   
\begin{aligned}
\max_{x \in \mathbb{R}^n, X\in \mathbb{S}^n} & \left \langle C,X \right \rangle \\
    \text{s.t. } & \ Diag(X) =  x, \\ 
    & \ \sum_{i=1}^n w_ix_i \leq c,\\
    & \ X-xx^\top  \succeq 0, & \\
    & \ x \in \{0,1\}.
\end{aligned}
\end{gather}

\section{Experimental results}\label{oa_sec:5}

In this section, we present experiments to evaluate the quality of our proposed algorithms, which are designed to tackle the special structure of binary semidefinite programs arising as exact formulations of binary quadratic problems. To simplify, we denote algorithm \ref{alg_spec_soc_OA} as \texttt{OA\_SOC} and algorithm \ref{alg_lazy_soc} as \texttt{LAZY\_SOC}.
We compare our algorithms with the majority of mixed-integer semidefinite program solvers that are known to us by means of evaluating the solving time on instances of the problems introduced in Section \ref{oa_sec:4}. These algorithms are the following. Since the generation of cuts $ v^\top \left  (X-xx^\top\right)v \geq 0$ does not rely on the OA iterations, they can be regarded as a strengthened formulation and also applied to other mixed-integer semidefinite program solvers. We denote the problems with the cuts as basic formulations, and those without the cuts are denoted as spectral formulations.

\begin{enumerate}
    \item The pure Outer Approximation algorithm\footnote{ \url{https://github.com/jump-dev/Pajarito.jl}} proposed by \cite{lubin2018polyhedral} and implemented in \texttt{PAJARITO}, denoted as \texttt{PAJARITO\_OA}-basic and \texttt{PAJARITO\_OA}-spectral.

    \item The Branch-and-Bound-outer-approximation extension of \texttt{PAJARITO} proposed in \cite{coey2020outer}, denoted as \texttt{PAJARITO\_TREE}-basic and \texttt{PAJARITO\_TREE}-spectral.

    \item The cutting plane algorithm of Kobayashi and Takano \cite{kobayashi2020branch}, denoted as \texttt{KOB\_Cutting\_Plane}.

    \item The Branch-and-Cut algorithm of \citet{kobayashi2020branch}, denoted as \texttt{KOB\_B\&C}.

    \item The Branch-and-Bound method implemented in \texttt{SCIP-SDP}\footnote{\url{https://www.opt.tu-darmstadt.de/scipsdp/}} \cite{gally2018framework}, denoted as \texttt{SCIPSDP}-basic and \texttt{SCIPSDP}-spectral.
\end{enumerate}

The pure outer approximation algorithm implemented in \texttt{PAJARITO} is similar to Algorithm \ref{alg_spec_soc_OA}, as the polyhedral outer approximations are updated in the same fashion. The main difference is that \texttt{PAJARITO}'s default version of the outer approximation algorithm solves mixed-integer linear problems rather than mixed-integer second-order cone problems. The Branch-and-Bound-Outer-Approximation extension of the algorithm proposed in \cite{coey2020outer} is more sophisticated and works by maintaining a single branch tree and inner and outer approximating the problems at each node of the tree. CUTSDP works by iteratively generating valid constraints for the SDP constraint and then by solving mixed-integer linear problems. If the solution for an integer linear problem is not PSD, an eigenvector corresponding to a negative eigenvalue is computed and added as a constraint in the problem. This algorithm is very similar to the cutting plane of Kobayashi and Takano. The Branch-and-cut algorithm by the same authors is similar to Algorithm \ref{alg_lazy_soc}, but without the initialization steps $1$ and $2$. Finally, \texttt{SCIP-SDP} implements a generic branch-and-bound approach where the subproblems inherit strong duality and the Slater condition of the semidefinite relaxations. To enhance performance, several solver components are incorporated, including dual fixing, branching rules, and primal heuristics.

To compare the different algorithms, we implemented Algorithm \ref{alg_spec_soc_OA} and \ref{alg_lazy_soc} in \texttt{Julia} and used the package \texttt{PAJARITO} directly. We implemented the two algorithms of \cite{kobayashi2020branch}. Since CUTSDP is essentially the same algorithm as the cutting plane algorithm of these two authors, we do not test the latter algorithm. We directly use \texttt{SCIPSDP} to solve the problems in the Conic Benchmark Format(CBF).

The main metric used to compare the algorithms is the time taken to solve the optimization problems. For each combination of parameters (which we describe in the following subsections), we generate $10$ random instances. Following Gally \cite{gally2018framework}, we use the \emph{shifted geometric time} to aggregate the times taken by each algorithm to solve the $10$ instances. The shifted geometric mean of values $y_1,\dots,y_n$ is defined as 

\[
\left (\prod_{j=1}^n(y_j+s)\right ) ^{\frac{1}{n}}-s
\]
where the shift $s$ was set to $10$. This aggregation metric intends to reduce the impact of easier instances. Each algorithm was given $1$ hour per instance, and $10$ random instances were generated for each configuration of parameters.

To solve the optimization sub-problems, we have used \texttt{MOSEK} 10.2.1 \cite{mosek} for semidefinite programs and \texttt{Gurobi} 11.0.2 \cite{optimization2020gurobi} for the mixed integer programs. Additionally, \texttt{SCIP} 4.3.0 was utilized in combination with \texttt{MOSEK} 10.2.1. The time limit for each experiment is set to 1 hour, and the thread limit is restricted to 1. All the other solver options are left at their default setting. The code used is available at
\url{https://github.com/SECQUOIA/Spectral-Outer-Approximation}.\footnote{The experiments were performed on a Linux cluster with 48 AMD EPYC 7643
2.3GHz CPUs and 1 TB RAM}

\subsection{Cardinality-constrained Boolean least squares}

We generate random instances of problem \eqref{BLS_SDP} by setting $b =0$ and taking $A$ to have either binary entries sampled uniformly and independently at random or entries sampled from the standard normal distribution. We set $A$ to be a $10 \times n$ matrix, and we vary the size of $n$. Observe that the size of the problem only depends on $n$. Finally, we set $k$ to be either $3$ or $5$ whenever $A$ has normally distributed entries and $k= 8,12 $ whenever $A$ has binary entries. These instances are considered in \cite{gally2016computing} and \cite{kobayashi2020branch} in their computations of restricted isometry constants. For each pair of $n$ and $k$, we randomly generate 10 instances and apply all the methods to solve them. We present the time performance of these methods in Figure \ref{fig:bls_normal} and \ref{fig:bls_bin}. The shifted geometric mean (SGM), with a shift value set to 10, is further analyzed in Tables \ref{tab:bls normal} and \ref{tab:bls binary} with time averaged over 10 instances. 
In these tables, ``\#Abort'' indicates the number of instances where the solver terminated unexpectedly, and ``\#Limit" means the number of instances where the solver reached the time limit.

The results of the cardinality-constrained Boolean least squares problem are consistent for both normally distributed entries and binary entries.
The \texttt{OA\_SOC} method with the basic formulation performs the best, even outperforming the \texttt{SCIPSDP} solver. Spectral cuts are only beneficial to the OA method in the \texttt{PAJARITO} solver, and the spectral cuts reduce the performance of all the other methods.

The results of the cardinality-constrained Boolean least squares problem are consistent for both normally distributed and binary entries. The \texttt{OA\_SOC} method, using the basic formulation, achieves the best performance, even surpassing even the \texttt{SCIPSDP} solver. Spectral cuts are only beneficial to the OA method in the \texttt{PAJARITO} solver, while for all other methods, spectral cuts tend to reduce performance. However, the OA method in the \texttt{PAJARITO} solver is not as good as the \texttt{OA\_SOC} method, even with the help of spectral cuts.
Additionally, the \texttt{LAZY\_SOC} method outperforms both the \texttt{KOB\_B\&C} and \texttt{KOB\_Cutting\_Plane} methods. For both \texttt{PAJARITO} and Kobayashi's algorithms, their single-tree implementations exhibit greater efficiency compared to their multi-tree counterparts.

\begin{figure}
    \centering
    \includegraphics[width=\textwidth]{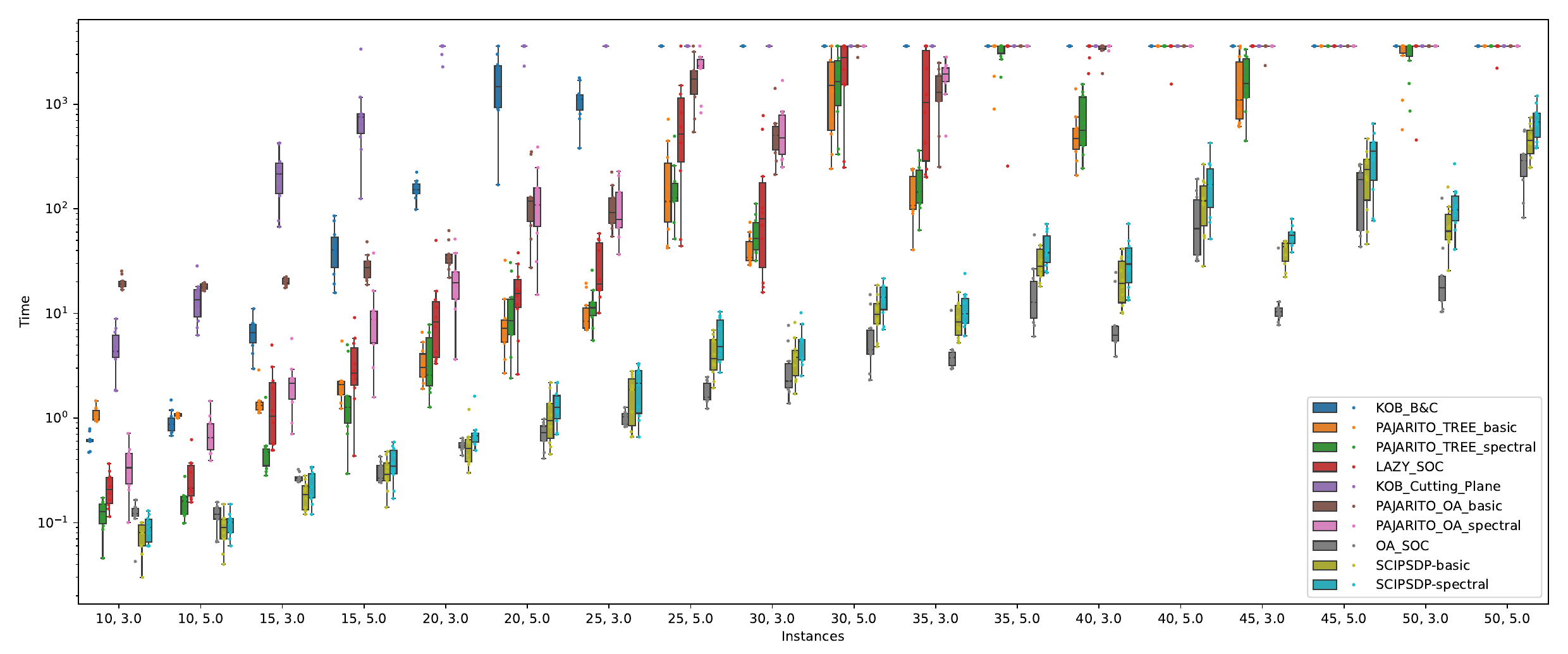}
    \caption{Time performance of different algorithms on the cardinality-constrained binary least squares problem with normal entries.}
    \label{fig:bls_normal}
\end{figure}

\begin{figure}
    \centering
    \includegraphics[width=\linewidth]{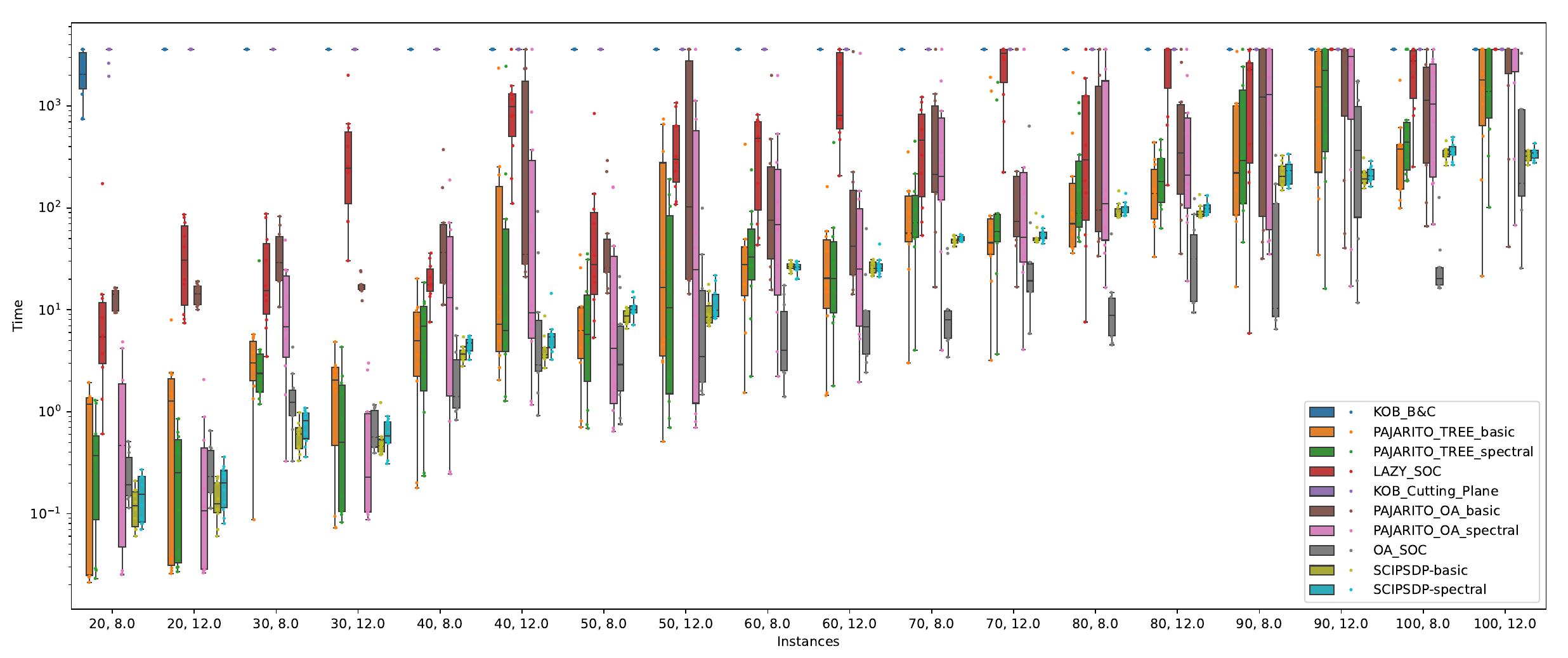}
    \caption{Time performance of different algorithms on the cardinality-constrained binary least squares problem with binary entries.}
    \label{fig:bls_bin}
\end{figure}

\begin{sidewaystable}
\centering
\caption{Performance of the different algorithms on the cardinality-constrained binary least squares problem with normally distributed entries.}
\label{tab:bls normal}
\resizebox{\textwidth}{!}{
\begin{tabular}{llllllllllll}
\toprule
 &  & \multicolumn{10}{c}{SGM (\#Abort, \#Limit)} \\
 & Algorithm & \texttt{KOB\_B\&C} & \texttt{KOB\_Cutting\_Plane} & \texttt{LAZY\_SOC} & \texttt{OA\_SOC} & \texttt{PAJARITO\_OA}-basic & \texttt{PAJARITO\_OA}-spectral & \texttt{PAJARITO\_TREE}-basic & \texttt{PAJARITO\_TREE}-spectral & \texttt{SCIPSDP}-basic & \texttt{SCIPSDP}-spectral \\
NumVar & Cardinality &  &  &  &  &  &  &  &  &  &  \\
\midrule
\multirow[t]{2}{*}{10} & 3 & 0.62 (0, 0) & 4.62 (0, 0) & 0.22 (0, 0) & 0.12 (0, 0) & 19.75 (0, 0) & 0.36 (0, 0) & 1.07 (0, 0) & 0.12 (0, 0) & 0.07 (0, 0) & 0.09 (0, 0) \\
 & 5 & 0.93 (0, 0) & 13.33 (0, 0) & 0.28 (0, 0) & 0.12 (0, 0) & 18.08 (0, 0) & 0.74 (0, 0) & 1.06 (0, 0) & 0.16 (0, 0) & 0.09 (0, 0) & 0.1 (0, 0) \\
\cline{1-12}
\multirow[t]{2}{*}{15} & 3 & 6.46 (0, 0) & 193.03 (0, 0) & 1.57 (0, 0) & 0.27 (0, 0) & 20.07 (0, 0) & 2.19 (0, 0) & 1.43 (0, 0) & 0.53 (0, 0) & 0.19 (0, 0) & 0.23 (0, 0) \\
 & 5 & 39.84 (0, 0) & 683.54 (0, 0) & 3.32 (0, 0) & 0.3 (0, 0) & 27.83 (0, 0) & 9.16 (0, 0) & 2.19 (0, 0) & 1.72 (0, 0) & 0.31 (0, 0) & 0.37 (0, 0) \\
\cline{1-12}
\multirow[t]{2}{*}{20} & 3 & 152.57 (0, 0) & 3375.69 (0, 8) & 9.81 (0, 0) & 0.55 (0, 0) & 34.98 (0, 0) & 19.61 (0, 0) & 3.4 (0, 0) & 3.55 (0, 0) & 0.55 (0, 0) & 0.74 (0, 0) \\
 & 5 & 1316.57 (0, 1) & 3444.74 (0, 9) & 15.12 (0, 0) & 0.71 (0, 0) & 112.6 (0, 0) & 101.66 (0, 0) & 8.26 (0, 0) & 10.32 (0, 0) & 1.1 (0, 0) & 1.34 (0, 0) \\
\cline{1-12}
\multirow[t]{2}{*}{25} & 3 & 1030.82 (0, 0) & 3600+ (0, 10) & 26.05 (0, 0) & 1.0 (0, 0) & 101.03 (0, 0) & 95.46 (0, 0) & 10.02 (0, 0) & 11.75 (0, 0) & 1.66 (0, 0) & 2.0 (0, 0) \\
 & 5 & 3600+ (0, 10) & 3600+ (0, 10) & 461.48 (0, 1) & 1.77 (0, 0) & 1585.07 (0, 1) & 2070.58 (0, 1) & 143.81 (0, 0) & 143.05 (0, 0) & 4.09 (0, 0) & 5.74 (0, 0) \\
\cline{1-12}
\multirow[t]{2}{*}{30} & 3 & 3600+ (0, 10) & 3600+ (0, 10) & 91.7 (0, 0) & 3.0 (0, 0) & 485.98 (0, 0) & 531.66 (0, 0) & 40.13 (0, 0) & 56.22 (0, 0) & 3.86 (0, 0) & 5.17 (0, 0) \\
 & 5 & 3600+ (0, 10) & 3600+ (0, 10) & 1716.24 (0, 4) & 5.87 (0, 0) & 3511.11 (0, 9) & 3600+ (0, 10) & 1177.74 (0, 2) & 1400.51 (0, 2) & 9.99 (0, 0) & 13.69 (0, 0) \\
\cline{1-12}
\multirow[t]{2}{*}{35} & 3 & 3600+ (0, 10) & 3600+ (0, 10) & 934.21 (0, 3) & 4.23 (0, 0) & 1155.55 (0, 0) & 1713.78 (0, 0) & 122.33 (0, 0) & 158.43 (0, 0) & 8.96 (0, 0) & 11.08 (0, 0) \\
 & 5 & 3600+ (0, 10) & 3600+ (0, 10) & 2771.71 (0, 9) & 15.5 (0, 0) & 3600+ (0, 10) & 3600+ (0, 10) & 2933.6 (0, 8) & 3194.13 (0, 7) & 29.5 (0, 0) & 40.58 (0, 0) \\
\cline{1-12}
\multirow[t]{2}{*}{40} & 3 & 3600+ (0, 10) & 3600+ (0, 10) & 3302.13 (0, 8) & 8.31 (0, 0) & 3326.78 (0, 7) & 3562.1 (0, 9) & 488.02 (0, 0) & 637.62 (0, 0) & 20.46 (0, 0) & 30.16 (0, 0) \\
 & 5 & 3600+ (0, 10) & 3600+ (0, 10) & 3311.5 (0, 9) & 68.92 (0, 0) & 3600+ (0, 10) & 3600+ (0, 10) & 3600+ (0, 10) & 3600+ (0, 10) & 105.97 (0, 0) & 157.81 (0, 0) \\
\cline{1-12}
\multirow[t]{2}{*}{45} & 3 & 3600+ (0, 10) & 3600+ (0, 10) & 3600+ (0, 10) & 12.07 (0, 0) & 3449.02 (0, 9) & 3600+ (0, 10) & 1315.35 (0, 1) & 1560.8 (0, 0) & 37.98 (0, 0) & 55.14 (0, 0) \\
 & 5 & 3600+ (0, 10) & 3600+ (0, 10) & 3600+ (0, 10) & 129.61 (0, 0) & 3600+ (0, 10) & 3600+ (0, 10) & 3600+ (0, 10) & 3600+ (0, 10) & 186.06 (0, 0) & 274.54 (0, 0) \\
\cline{1-12}
\multirow[t]{2}{*}{50} & 3 & 3600+ (0, 10) & 3600+ (0, 10) & 2931.34 (0, 9) & 22.49 (0, 0) & 3600+ (0, 10) & 3600+ (0, 10) & 2604.42 (0, 6) & 2785.48 (0, 7) & 65.87 (0, 0) & 100.46 (0, 0) \\
 & 5 & 3600+ (0, 10) & 3600+ (0, 10) & 3429.32 (0, 9) & 258.98 (0, 0) & 3600+ (0, 10) & 3600+ (0, 10) & 3600+ (0, 10) & 3600+ (0, 10) & 434.39 (0, 0) & 660.92 (0, 0) \\
\cline{1-12}
\bottomrule
\end{tabular}}
\end{sidewaystable}

\begin{sidewaystable}
\centering
\caption{Performance of the different algorithms on the cardinality-constrained binary least squares problem with binary entries.}
\label{tab:bls binary}
\resizebox{\textwidth}{!}{
\begin{tabular}{llllllllllll}
\toprule
 &  & \multicolumn{10}{r}{SGM (\#Abort, \#Limit)} \\
 & Algorithm & \texttt{KOB\_B\&C} & \texttt{KOB\_Cutting\_Plane} & \texttt{LAZY\_SOC} & \texttt{OA\_SOC} & \texttt{PAJARITO\_OA}-basic & \texttt{PAJARITO\_OA}-spectral & \texttt{PAJARITO\_TREE}-basic & \texttt{PAJARITO\_TREE}-spectral & \texttt{SCIPSDP}-basic & \texttt{SCIPSDP}-spectral \\
NumVar & Cardinality &  &  &  &  &  &  &  &  &  &  \\
\midrule
\multirow[t]{2}{*}{20} & 8 & 1935.09 (0, 3) & 3282.64 (0, 8) & 9.82 (0, 0) & 0.25 (0, 0) & 12.87 (0, 0) & 1.24 (0, 0) & 0.83 (0, 0) & 0.46 (0, 0) & 0.12 (0, 0) & 0.16 (0, 0) \\
 & 12 & 3600+ (0, 10) & 3600+ (0, 10) & 30.31 (0, 0) & 0.29 (0, 0) & 14.1 (0, 0) & 0.38 (0, 0) & 1.5 (0, 0) & 0.31 (0, 0) & 0.14 (0, 0) & 0.2 (0, 0) \\
\cline{1-12}
\multirow[t]{2}{*}{30} & 8 & 3600+ (0, 10) & 3600+ (0, 10) & 22.3 (0, 0) & 1.46 (0, 0) & 32.81 (0, 0) & 10.02 (0, 0) & 3.1 (0, 0) & 3.9 (0, 0) & 0.59 (0, 0) & 0.76 (0, 0) \\
 & 12 & 3600+ (0, 10) & 3600+ (0, 10) & 250.75 (0, 0) & 0.7 (0, 0) & 17.27 (0, 0) & 0.78 (0, 0) & 1.8 (0, 0) & 1.06 (0, 0) & 0.53 (0, 0) & 0.61 (0, 0) \\
\cline{1-12}
\multirow[t]{2}{*}{40} & 8 & 3600+ (0, 10) & 3600+ (0, 10) & 19.28 (0, 0) & 2.55 (0, 0) & 45.68 (0, 0) & 18.26 (0, 0) & 5.52 (0, 0) & 6.05 (0, 0) & 3.74 (0, 0) & 4.55 (0, 0) \\
 & 12 & 3600+ (0, 10) & 3600+ (0, 10) & 766.49 (0, 1) & 8.18 (0, 0) & 133.19 (0, 1) & 53.14 (0, 1) & 34.31 (0, 0) & 29.14 (0, 0) & 4.19 (0, 0) & 5.68 (0, 0) \\
\cline{1-12}
\multirow[t]{2}{*}{50} & 8 & 3600+ (0, 10) & 3600+ (0, 10) & 45.18 (0, 0) & 5.09 (0, 0) & 42.62 (0, 0) & 14.4 (0, 0) & 8.07 (0, 0) & 8.43 (0, 0) & 8.61 (0, 0) & 10.4 (0, 0) \\
 & 12 & 3600+ (0, 10) & 3600+ (0, 10) & 338.05 (0, 0) & 9.99 (0, 0) & 177.77 (0, 3) & 60.35 (0, 1) & 41.51 (0, 0) & 20.52 (0, 0) & 9.68 (0, 0) & 11.8 (0, 0) \\
\cline{1-12}
\multirow[t]{2}{*}{60} & 8 & 3600+ (0, 10) & 3600+ (0, 10) & 273.0 (0, 0) & 5.75 (0, 0) & 109.66 (0, 0) & 78.55 (0, 0) & 30.18 (0, 0) & 35.79 (0, 0) & 26.6 (0, 0) & 25.6 (0, 0) \\
 & 12 & 3600+ (0, 10) & 3600+ (0, 10) & 1126.34 (0, 3) & 9.45 (0, 0) & 74.6 (0, 0) & 43.14 (0, 0) & 23.87 (0, 0) & 27.84 (0, 0) & 25.84 (0, 0) & 26.92 (0, 0) \\
\cline{1-12}
\multirow[t]{2}{*}{70} & 8 & 3600+ (0, 10) & 3600+ (0, 10) & 329.35 (0, 0) & 10.52 (0, 0) & 296.73 (0, 1) & 241.28 (0, 1) & 65.53 (0, 0) & 79.32 (0, 0) & 47.83 (0, 0) & 50.17 (0, 0) \\
 & 12 & 3600+ (0, 10) & 3600+ (0, 10) & 2016.43 (0, 5) & 29.85 (0, 0) & 157.94 (0, 2) & 120.96 (0, 2) & 82.69 (0, 0) & 92.39 (0, 0) & 53.09 (0, 0) & 54.78 (0, 0) \\
\cline{1-12}
\multirow[t]{2}{*}{80} & 8 & 3600+ (0, 10) & 3600+ (0, 10) & 299.29 (0, 1) & 10.82 (0, 0) & 241.59 (0, 2) & 214.53 (0, 2) & 114.99 (0, 0) & 149.26 (0, 0) & 93.03 (0, 0) & 97.04 (0, 0) \\
 & 12 & 3600+ (0, 10) & 3600+ (0, 10) & 1927.06 (0, 7) & 31.5 (0, 0) & 394.85 (0, 1) & 286.21 (0, 1) & 133.87 (0, 0) & 186.58 (0, 0) & 91.3 (0, 0) & 101.33 (0, 0) \\
\cline{1-12}
\multirow[t]{2}{*}{90} & 8 & 3600+ (0, 10) & 3600+ (0, 10) & 804.34 (0, 1) & 35.26 (0, 0) & 566.42 (0, 4) & 561.08 (0, 4) & 276.9 (0, 0) & 391.78 (0, 1) & 211.44 (0, 0) & 219.96 (0, 0) \\
 & 12 & 3600+ (0, 10) & 3600+ (0, 10) & 3600+ (0, 10) & 253.1 (0, 0) & 1125.75 (0, 5) & 1003.52 (0, 5) & 791.5 (0, 2) & 921.31 (0, 3) & 200.45 (0, 0) & 212.12 (0, 0) \\
\cline{1-12}
\multirow[t]{2}{*}{100} & 8 & 3600+ (0, 10) & 3600+ (0, 10) & 1829.16 (0, 3) & 26.38 (0, 0) & 741.5 (0, 2) & 726.15 (0, 1) & 310.66 (0, 0) & 550.73 (0, 2) & 339.35 (0, 0) & 360.34 (0, 0) \\
 & 12 & 3600+ (0, 10) & 3600+ (0, 10) & 3600+ (0, 7) & 310.01 (0, 0) & 1692.22 (0, 7) & 1774.35 (0, 7) & 1051.71 (0, 4) & 1244.98 (0, 4) & 319.08 (0, 0) & 339.83 (0, 0) \\
\cline{1-12}
\bottomrule
\end{tabular}}
\end{sidewaystable}

\subsection{Quadratic knapsack problem}

We generate random instances of the quadratic knapsack problem following \cite{pisinger2007quadratic}, which specify instances that have become the standard to test this optimization problem computationally. That is, we first set a \emph{density} value $\Delta = 0.8$, which corresponds to the percentage of nonzero elements of the matrix $C$. Each weight $w_j , \ j \in [n]$ is uniformly randomly distributed in $[1,50]$. The entry $ij$ of $C$ is equal to the entry $ji$ and is nonzero with probability $\Delta$, in which case it is uniformly distributed in $[1,100], \ i,j \in [n]$. The capacity $c$ of the knapsack is fixed at $\frac{1}{2}\sum_{j=1}^n w_j$. We randomly generate 10 instances for each setting and apply all the methods to solve them. The results are presented in Figure \ref{fig:knapsack} and in Table \ref{tab:knapsack}.

For the quadratic knapsack problem, \texttt{SCIPSDP} with the basic formulation performs the best among all tested methods, while adding spectral cuts reduces its efficiency. Additionally, the \texttt{OA\_SOC} method consistently outperforms the OA method in the \texttt{PAJARITO} solver, regardless of whether spectral cuts are included. Spectral cuts provide a performance boost to both single-tree and multi-tree implementations in the \texttt{PAJARITO} solver. The \texttt{LAZY\_SOC} method also demonstrates superior performance compared to the \texttt{KOB\_B\&C} and \texttt{KOB\_Cutting\_Plane} methods. Overall, single-tree implementations are generally more efficient than their multi-tree counterparts, confirming their advantage across various methods.

\begin{figure}
    \centering
    \includegraphics[width=\textwidth]{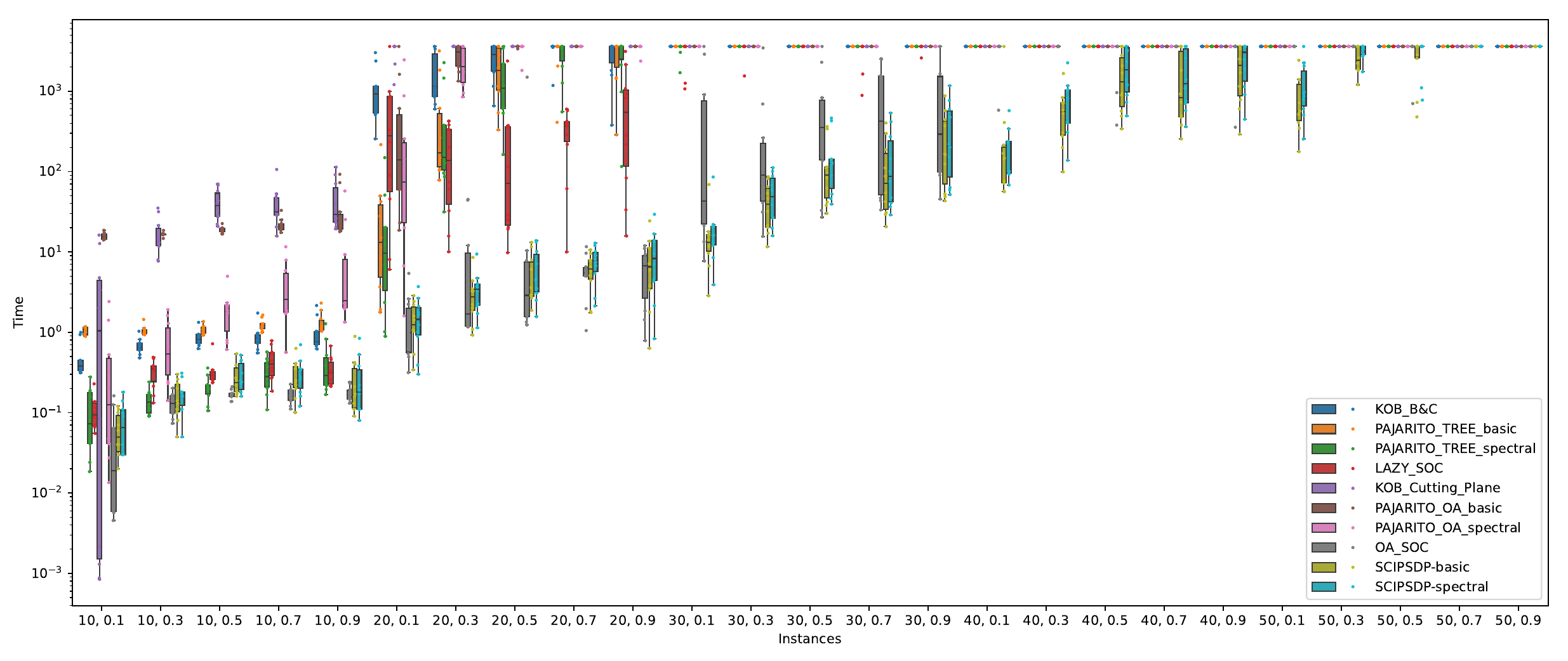}
    \caption{Time performance of different algorithms on the quadratic knapsack problem}
    \label{fig:knapsack}
\end{figure}

\begin{sidewaystable}
\caption{Performance of the different algorithms on the quadratic knapsack problem.}
\label{tab:knapsack}
\resizebox{\textwidth}{!}{
\begin{tabular}{llllllllllll}
\toprule
 &  & \multicolumn{10}{c}{SGM (\#Abort, \#Limit)} \\
 & Algorithm & \texttt{KOB\_B\&C} & \texttt{KOB\_Cutting\_Plane} & \texttt{LAZY\_SOC} & \texttt{OA\_SOC} & \texttt{PAJARITO\_OA}-basic & \texttt{PAJARITO\_OA}-spectral & \texttt{PAJARITO\_TREE}-basic & \texttt{PAJARITO\_TREE}-spectral & \texttt{SCIPSDP}-basic & \texttt{SCIPSDP}-spectral \\
NumVar & Density &  &  &  &  &  &  &  &  &  &  \\
\midrule
\multirow[t]{5}{*}{10} & 0.1 & 0.48 (0, 0) & 3.04 (0, 0) & 0.11 (0, 0) & 0.05 (0, 0) & 15.73 (0, 0) & 0.48 (0, 0) & 1.02 (0, 0) & 0.11 (0, 0) & 0.06 (0, 0) & 0.08 (0, 0) \\
 & 0.3 & 0.68 (0, 0) & 15.7 (0, 0) & 0.33 (0, 0) & 0.13 (0, 0) & 16.6 (0, 0) & 0.75 (0, 0) & 1.04 (0, 0) & 0.14 (0, 0) & 0.16 (0, 0) & 0.17 (0, 0) \\
 & 0.5 & 0.85 (0, 0) & 38.47 (0, 0) & 0.33 (0, 0) & 0.17 (0, 0) & 18.88 (0, 0) & 1.86 (0, 0) & 1.07 (0, 0) & 0.21 (0, 0) & 0.28 (0, 0) & 0.3 (0, 0) \\
 & 0.7 & 0.88 (0, 0) & 36.39 (0, 0) & 0.44 (0, 0) & 0.16 (0, 0) & 21.54 (0, 0) & 3.71 (0, 0) & 1.24 (0, 0) & 0.31 (0, 0) & 0.29 (0, 0) & 0.31 (0, 0) \\
 & 0.9 & 1.0 (0, 0) & 39.12 (0, 0) & 0.35 (0, 0) & 0.17 (0, 0) & 28.67 (0, 0) & 6.99 (0, 0) & 1.33 (0, 0) & 0.43 (0, 0) & 0.27 (0, 0) & 0.27 (0, 0) \\
\cline{1-12}
\multirow[t]{5}{*}{20} & 0.1 & 888.49 (0, 0) & 3069.42 (0, 8) & 216.03 (0, 1) & 1.46 (0, 0) & 185.21 (0, 1) & 94.06 (0, 0) & 19.53 (0, 0) & 14.53 (0, 0) & 1.43 (0, 0) & 1.53 (0, 0) \\
 & 0.3 & 1394.94 (0, 2) & 3600+ (0, 10) & 108.41 (0, 0) & 6.73 (0, 0) & 2624.21 (0, 4) & 1948.66 (0, 3) & 287.38 (0, 0) & 232.29 (0, 0) & 2.95 (0, 0) & 3.44 (0, 0) \\
 & 0.5 & 2244.16 (0, 5) & 3600+ (0, 10) & 103.81 (0, 0) & 11.89 (0, 0) & 3575.3 (0, 9) & 3361.46 (0, 9) & 1568.73 (0, 3) & 1087.55 (0, 1) & 5.2 (0, 0) & 5.89 (0, 0) \\
 & 0.7 & 3212.19 (0, 8) & 3600+ (0, 10) & 237.3 (0, 0) & 5.48 (0, 0) & 3600+ (0, 10) & 3600+ (0, 10) & 2742.44 (0, 8) & 2510.09 (0, 4) & 5.87 (0, 0) & 7.3 (0, 0) \\
 & 0.9 & 2475.14 (0, 7) & 3600+ (0, 10) & 352.36 (0, 0) & 5.82 (0, 0) & 3600+ (0, 10) & 3452.18 (0, 9) & 2337.15 (0, 7) & 2141.15 (0, 7) & 7.21 (0, 0) & 8.77 (0, 0) \\
\cline{1-12}
\multirow[t]{5}{*}{30} & 0.1 & 3600+ (0, 10) & 3600+ (0, 10) & 2870.81 (0, 8) & 134.85 (0, 1) & 3600+ (0, 10) & 3600+ (0, 10) & 3600+ (0, 10) & 3284.03 (0, 8) & 14.38 (0, 0) & 17.57 (0, 0) \\
 & 0.3 & 3600+ (0, 10) & 3600+ (0, 10) & 3308.75 (0, 9) & 127.82 (0, 0) & 3600+ (0, 10) & 3600+ (0, 10) & 3600+ (0, 10) & 3600+ (0, 10) & 37.16 (0, 0) & 46.84 (0, 0) \\
 & 0.5 & 3600+ (0, 10) & 3600+ (0, 10) & 3600+ (0, 10) & 320.05 (0, 1) & 3600+ (0, 10) & 3600+ (0, 10) & 3600+ (0, 10) & 3600+ (0, 10) & 93.46 (0, 0) & 120.96 (0, 0) \\
 & 0.7 & 3600+ (0, 10) & 3600+ (0, 10) & 2892.94 (0, 8) & 302.23 (0, 0) & 3600+ (0, 10) & 3600+ (0, 10) & 3600+ (0, 10) & 3600+ (0, 10) & 83.09 (0, 0) & 107.38 (0, 0) \\
 & 0.9 & 3600+ (0, 10) & 3600+ (0, 10) & 3483.75 (0, 9) & 391.96 (0, 2) & 3600+ (0, 10) & 3600+ (0, 10) & 3600+ (0, 10) & 3600+ (0, 10) & 182.71 (0, 0) & 230.71 (0, 0) \\
\cline{1-12}
\multirow[t]{5}{*}{40} & 0.1 & 3600+ (0, 10) & 3600+ (0, 10) & 3600+ (0, 10) & 3001.65 (0, 9) & 3600+ (0, 10) & 3600+ (0, 10) & 3600+ (0, 10) & 3600+ (0, 10) & 174.31 (1, 0) & 175.7 (0, 0) \\
 & 0.3 & 3600+ (0, 10) & 3600+ (0, 10) & 3600+ (0, 10) & 3600+ (0, 10) & 3600+ (0, 10) & 3600+ (0, 10) & 3600+ (0, 10) & 3600+ (0, 10) & 455.87 (0, 0) & 651.57 (0, 0) \\
 & 0.5 & 3600+ (0, 10) & 3600+ (0, 10) & 3600+ (0, 10) & 2522.9 (0, 8) & 3600+ (0, 10) & 3600+ (0, 10) & 3600+ (0, 10) & 3600+ (0, 10) & 1261.01 (0, 2) & 1689.45 (0, 3) \\
 & 0.7 & 3600+ (0, 10) & 3600+ (0, 10) & 3600+ (0, 10) & 3600+ (0, 10) & 3600+ (0, 10) & 3600+ (0, 10) & 3600+ (0, 10) & 3600+ (0, 10) & 1030.8 (0, 3) & 1350.7 (0, 3) \\
 & 0.9 & 3600+ (0, 10) & 3600+ (0, 10) & 3600+ (0, 10) & 2861.18 (0, 9) & 3600+ (0, 10) & 3600+ (0, 10) & 3600+ (0, 10) & 3600+ (0, 10) & 1598.22 (0, 3) & 2054.17 (0, 5) \\
\cline{1-12}
\multirow[t]{5}{*}{50} & 0.1 & 3600+ (0, 10) & 3600+ (0, 10) & 3600+ (0, 10) & 3600+ (0, 10) & 3600+ (0, 10) & 3600+ (0, 10) & 3600+ (0, 10) & 3600+ (0, 10) & 653.32 (0, 0) & 978.33 (1, 0) \\
 & 0.3 & 3600+ (0, 10) & 3600+ (0, 10) & 3600+ (0, 10) & 3600+ (0, 10) & 3600+ (0, 10) & 3600+ (0, 10) & 3600+ (0, 10) & 3600+ (0, 10) & 2438.64 (0, 4) & 3087.49 (1, 4) \\
 & 0.5 & 3600+ (0, 10) & 3600+ (0, 10) & 3600+ (0, 10) & 3058.57 (0, 9) & 3600+ (0, 10) & 3600+ (0, 10) & 3600+ (0, 10) & 3600+ (0, 10) & 2368.99 (0, 6) & 2744.46 (0, 8) \\
 & 0.7 & 3600+ (0, 10) & 3600+ (0, 10) & 3600+ (0, 10) & 3600+ (0, 10) & 3600+ (0, 10) & 3600+ (0, 10) & 3600+ (0, 10) & 3600+ (0, 10) & 3600+ (0, 10) & 3600+ (0, 10) \\
 & 0.9 & 3600+ (0, 10) & 3600+ (0, 10) & 3600+ (0, 10) & 3600+ (0, 10) & 3600+ (0, 10) & 3600+ (0, 10) & 3600+ (0, 10) & 3600+ (0, 10) & 3600+ (0, 10) & 3600+ (0, 10) \\
\cline{1-12}
\bottomrule
\end{tabular}}
\end{sidewaystable}

\section{Conclusion and future work}\label{oa_sec:6}

In this work, we have proposed two algorithms based on spectral decomposition to improve outer approximation algorithms for integer semidefinite programs whenever they are derived from binary QCQPs. The experiments evaluating our approach are promising and seem to indicate that integer semidefinite programming is a strong candidate for solving binary QCQPs, competing with state-of-the-art global solvers.

In a sense, these algorithms have been implemented in their most basic form, and we believe they can be improved using the Branch-and-Bound ideas of \citet{coey2020outer}. 
Our strengthenings were based on adding cuts derived from eigenvectors of a certain matrix, which simultaneously diagonalizes the matrix, determining the objective of an ISDP and aggregating the constraint matrices. We believe more progress can be made in this direction and that different alternatives to this approach could be found. 

Copositive and completely positive programming are also alternatives to solve BQCQPs. Although some algorithms that inner and outer approximate these cones have been proposed, such as the algorithm of \citet{bundfuss2009adaptive} 
and the recent work of \citet{gouveia2020inner}, the size of problems that can be solved remains limited. It would be interesting to find families of copositive and completely positive problems that can be cast as integer QCQPs and, therefore, could be tackled with integer semidefinite optimization problems, providing a new way to approach these hard conic problems.

\section*{Acknowledgements}
D.B.N. and Z.P. acknowledge the support of the startup grant of the Davidson School of Chemical Engineering at Purdue University.

\bibliographystyle{unsrtnat}
\bibliography{lit} 

\end{document}